\documentclass[11pt]{article}
\usepackage{graphicx}
\usepackage{graphics}
\usepackage{amsfonts}
\usepackage{amscd}
\usepackage{amsthm}
\usepackage{indentfirst}
\usepackage[all]{xy}
\usepackage{amssymb, amsmath, amsthm, amsgen, amstext, amsbsy, amsopn}
\usepackage{epic,eepic}
\usepackage{hyperref}
\usepackage{enumitem} 
\usepackage{color}
\usepackage{ dsfont }

\theoremstyle{plain}
\newtheorem{theorem}{Theorem}
\newtheorem{lemma}[theorem]{Lemma}

\newtheorem{cor}[theorem]{Corollary}
\newtheorem*{theorem*}{Theorem}

\theoremstyle{definition}

\newtheorem{rem}[theorem]{Remark}
\newtheorem{example}[theorem]{Example}

\newcommand{\R}{{\mathbb{R}}}

\newcommand{\C}{{\mathbb{C}}}

\newcommand{\cL}{{\mathcal{L}}}

\def\id{{1\hskip-2.5pt{\rm l}}}

\newcommand{\rank}{{ \operatorname{rank}}}
\newcommand{\Sp}{{ \operatorname{Sp}}}

\newcommand{\Span}{{ \operatorname{span}}}

\definecolor{lev}{rgb}{0.773,0.294,0.549}

\begin{document}

\title{Applications of Grothendieck's inequality to linear symplectic geometry}

\author{Efim Gluskin, Shira Tanny}

\maketitle

\begin{abstract}
      Recently in symplectic geometry there arose an interest in bounding various functionals on spaces of matrices. It appears that Grothendieck's theorems about factorization are a useful tool for proving such bounds. In this note we present two such applications. 
\end{abstract}

Linear symplectic geometry concerns a non-degenerate anti-symmetric bilinear form, called a {\it symplectic bilinear form}, and matrices that preserve this form, called {\it symplectic matrices}. Note that such forms exist only on even dimensional linear spaces. A standard example for this setting is the form defined on $\R^{2n}$ by $(u,v)\mapsto \left<u, Jv\right>$ where $J$ (or $J_{2n}$) is the $2n\times 2n$ matrix corresponding to multiplication by $i$ under the identification $\R^{2n}\cong \C^n$, 
\begin{equation*}
J:=\left(\begin{array}{cc}
0 &-\id_{n\times n} \\\id_{n\times n} & 0 \end{array}\right).
\end{equation*}
A matrix $S$ preserves this bilinear form if and only if it satisfies the following relation: 
\begin{equation}\label{eq:symp_matrix}
S^T JS = J.
\end{equation} 
The space of such matrices is denoted by $\Sp(2n)$. 

More generally, given an even dimensional Euclidean vector space $E$, we denote by $J$ or $J_E$ a linear orthogonal transformation whose square equals to minus the identity on $E$, $J^2=-\id_E$. 
In this case, matrices satisfying the relation (\ref{eq:symp_matrix}) with respect to $J_E$ preserve the form  $(u,v)\mapsto \left<u, J_E v\right>$ and the space of such matrices is denoted by $\Sp(E;J_E)$.

In \cite[Appendix A]{buhovsky2017poisson}, Buhovsky, Logunov and Tanny proved that there exists a constant $c(n)$, depending on the dimension, such that for any finite collection of vectors $v_1,\dots, v_N\in \R^{2n}$,
\begin{equation}\label{eq:BLT_ineq}
\sum_{i,j=1}^N |\left<v_i, Jv_j\right>|\leq c(n)\cdot \max_{|s_i|,|t_j|\leq 1}\left<\sum_{i=1}^N t_iv_i, J\sum_{j=1}^N s_j v_j\right>,
\end{equation}
where the constant $c(n)$ grows exponentially in $n$. 
In this paper it is shown that, using Grothendieck's inequality \cite{grothendieck1956resume}, the growth of the constant $c(n)$ in the above inequality can be improved to be $\sqrt n$. This result is stated in Corollary~\ref{cor:symp_ineq} below, and in Theorem~\ref{thm:main_ineq} in a more general setting. Example~\ref{exa:sqrt_sharp} shows that the growth of $c(n) \propto \sqrt{n}$ is sharp. The second main result concerns the orbit of a finite collection of vectors under the action of the group of symplectic matrices. Grothendieck's inequality can be used to prove a sharp upper-bound for the minimal sum of norms of given vectors under the action of symplectic matrices, as stated in Theorem~\ref{thm:norm_under_Sp} below. 

We use some basic facts and notations from operator theory. Denote by $\ell_p^N$ the space $\R^N$ equipped with the norm:  
\begin{eqnarray*}
\|u\|_{\ell_p^N} &:=& \left(\sum_{i=1}^{N} |u_i|^p\right)^\frac{1}{p}, \quad 1\leq p<\infty\\
\|u\|_{\ell_\infty^N} &:=& \max_{1\leq i\leq N} |u_i|,
\end{eqnarray*}
for $u=(u_1,\dots,u_N)\in \R^N$.
The space of linear operators from $\R^n$ to $\R^m$ is denoted by $\cL(\R^n,\R^m)$. We identify an operator $A\in\cL(\R^n,\R^m)$ with its matrix $A=(a_{ij})_{i=1}^m,_{j=1}^n$. For such a matrix, one denotes by $A^T\in\cL(\R^m, \R^n)$ its transpose. For a vector $\Lambda=(\lambda_1,\dots,\lambda_n)\in\R^n$ we denote by $D_\Lambda$ the diagonal matrix corresponding to $\Lambda$, namely, $(D_\Lambda)_{ii}=\lambda_i$ and $(D_\Lambda)_{ij}=0$ for $i\neq j$.

For $1\leq p,q\leq\infty$ one denotes by $\cL(\ell_p^n,\ell_q^m)$ the linear space 	$\cL(\R^n, \R^m)$ equipped with the operator norm from $\ell_p^n$ to $\ell_q^m$:
\begin{equation*}
\|A\|_{\cL(\ell_p^n,\ell_q^m)}=\max_{\|u\|_{\ell_p^n} =1} \|Au\|_{\ell_q^m}.
\end{equation*}
For $A\in\cL(\R^n,\R^m)$ let $\lambda_0\geq\lambda_1 \geq\cdots\geq\lambda_{n-1}$ be the sequence of all eigenvalues $\lambda_j = \lambda_j(A^T A)$ of the operator $A^TA$ with multiplicities. The $j$-singular value of $A$ is defined as
\begin{equation*}
s_j(A) :=\begin{cases}\sqrt{\lambda_j(A^TA)}, &j=0,\dots,n-1,\\
				0, &  j\geq n.\end{cases}
\end{equation*}
Recall that $s_0(A)$ coincides with the operator norm $\|A\|_{\cL(\ell_2^n,\ell_2^m)}$.
The Hilbert-Schmidt norm of $A\in \cL(\R^n, \R^m)$ is defined as
\begin{equation}
\|A\|_{HS} := \sqrt{\sum_{j=0}^\infty s_j(A)^2}.
\end{equation}
If $(a_{ij})_{i=1}^m,_{j=1}^n$ is the matrix representing $A$ then
\begin{equation}
\|A\|_{HS} = \sqrt{\sum_{i=1}^m\sum_{j=1}^n a_{ij}^2}.
\end{equation}

For an operator $A\in\cL(\R^n,\R^m)$ such that $\rank(A) =k$, one has $s_j(A)=0$ for all $j\geq k$. Consequently,
\begin{equation}\label{eq:norms_rk_ineq}
\|A\|_{HS} = \sqrt{\sum_{j=0}^{k-1}s_j(A)^2}\leq \sqrt{k}\cdot s_0(A) =\sqrt{\rank(A)}\cdot \|A\|_{\cL(\ell_2^n,\ell_2^m)}
\end{equation}

The following reformulation of Grothendieck's theorem \cite{grothendieck1956resume} is a special case of Theorem 2.1 in \cite{pisier2012grothendieck}, where the compact sets $S$ and $T$ are finite.

\begin{theorem*}[see \cite{pisier2012grothendieck}]
	Let $(a_{ij})_{i=1}^m,_{j=1}^n$ be an $m\times n$ real matrix. Then, there exist vectors ${\Lambda_0}=(\lambda_1^0,\dots,\lambda_m^0)$ and $\Lambda_1=(\lambda_1^1,\dots,\lambda_n^1)$ with non-negative entries $\lambda_i^0\geq 0$, $\lambda_i^1\geq 0$ and Euclidean norms bounded by 1,  $\|\Lambda_0\|_{\ell_2^m}\leq 1$, $\|\Lambda_1\|_{\ell_2^n}\leq 1$,  and there exists an $m\times n$ matrix $B$ such that 	
	\begin{equation}
	A = D_{\Lambda_0} B D_{\Lambda_1}
	\end{equation}
	and 
	\begin{equation}
	\|B\|_{\cL(\ell_2^n,\ell_2^m)} \leq K\cdot \|A\|_{\cL(\ell_\infty^n, \ell_1^m)}
	\end{equation}
	where $K$ is an absolute constant.
\end{theorem*} 
\noindent The smallest value of the constant $K$ is called Grothendieck's constant and is denoted by $K_G$. Its exact value is still unknown, Grothendieck himself proved that $\pi/2\leq K_G\leq \sinh(\pi/2)$. The following consequence of Grothendieck's inequality is well known to the experts. 
\begin{lemma}\label{lem:our_GT}
	For any $A\in \cL(\R^n, \R^n)$ there exists a vector with strictly positive coordinates $\Lambda = (\lambda_1,\dots,\lambda_n)\in \R^n$, $\lambda_i>0$ for all $1\leq i\leq n$, such that $\|\Lambda\|_{\ell_2^n}\leq 1$ and 
	\begin{equation}\label{eq:our_GT}
	\|D_\Lambda^{-1}AD_\Lambda^{-1}\|_{\cL(\ell_2^n,\ell_2^n)}\leq 3K_G\cdot \|A\|_{\cL(\ell_\infty^n,\ell_1^n)}.
	\end{equation}
\end{lemma}
\begin{proof}
	Let $\Lambda_0 = (\lambda_1^0,\dots,\lambda_n^0)$, $\Lambda_1 = (\lambda_1^1,\dots,\lambda_n^1)$ and $B$ be the vectors and matrix from Grothendieck's theorem. Let us define 
	\begin{equation*}
	\lambda_i:=\frac{1}{\sqrt 3}\left(\frac{\sqrt 3-\sqrt 2}{\sqrt n} +\max\{\lambda_i^0, \lambda_i^1\}\right)>0.
	\end{equation*}
	Then, for $\Lambda=(\lambda_1,\dots,\lambda_n)$, it is clear that $\|\Lambda\|_{\ell_2^n}\leq 1$ and that $\|D_\Lambda^{-1} D_{\Lambda_j}\|_{\cL(\ell_2^n,\ell_2^n)}\leq \sqrt 3$. Inequality (\ref{eq:our_GT}) easily follows.  
\end{proof}

The following result provides an asymptotically sharp bound for the constant $c(n)$ from (\ref{eq:BLT_ineq}).
\begin{theorem}\label{thm:main_ineq}
	For any $N\times N$ matrix $A=(a_{ij})_{i,j=1}^N$,
	\begin{equation}\label{eq:main_ineq}
	\sum_{i,j=1}^N |a_{ij}| \leq 3 K_G \cdot \sqrt{\rank A}\cdot \|A\|_{\cL(\ell_\infty^N,\ell_1^N)}.
	\end{equation}
\end{theorem}
\begin{proof}
	Let $\Lambda=(\lambda_1,\dots,\lambda_n)$ be the vector from Lemma~\ref{lem:our_GT} corresponding to the matrix $A$.
	Denoting $B:=D_\Lambda^{-1} A D_\Lambda^{-1}$, $b_{ij}=\frac{1}{\lambda_i\lambda_j}a_{ij}$, we have
	\begin{equation*}
	\sum_{i,j=1}^N|a_{ij}| = \sum_{i,j=1}^N |b_{ij}|\cdot|\lambda_i\lambda_j|. 
	\end{equation*}
	By the Cauchy-Schwarz inequality,
	\begin{eqnarray*}
		\sum_{i,j=1}^N |b_{ij}|\cdot|\lambda_i\lambda_j|&\leq& 
		\left(\sum_{i,j=1}^N b_{ij}^2\right)^{\frac{1}{2}}\cdot \left(\sum_{i,j=1}^N \lambda_i^2\lambda_j^2 \right)^{\frac{1}{2}}\\
		&\leq& \left(\sum_{i,j=1}^N b_{ij}^2\right)^{\frac{1}{2}}=\|B\|_{HS},	
	\end{eqnarray*}
	and by (\ref{eq:norms_rk_ineq}) one has $\|B\|_{HS}\leq \sqrt{\rank B}\cdot \|B\|_{\cL(\ell_2^N,\ell_2^N)}$. Since $B$ is a multiplication of $A$ by invertible matrices, $\rank B=\rank A$. Moreover, by Lemma~\ref{lem:our_GT} the $\cL(\ell_2^N,\ell_2^N)$-norm of $B$ is bounded by $3K_G\cdot \|A\|_{\cL(\ell_\infty^N,\ell_1^N)}$. Overall we obtain
	\begin{equation*}
	\sum_{i,j=1}^N |a_{ij}|\leq \|B\|_{HS}\leq\sqrt{\rank B}\cdot \|B\|_{\cL(\ell_2^N,\ell_2^N)}\leq  \sqrt{\rank A}\cdot 3K_G\cdot \|A\|_{\cL(\ell_\infty^N,\ell_1^N)}
	\end{equation*} 
\end{proof}

The following example shows that the dependence on $\rank A$ in Theorem~\ref{thm:main_ineq} is sharp.
\begin{example} \label{exa:sqrt_sharp}
Let $k\leq N$ and take 
$$
A = \left(\begin{array}{cc}
	U_{k\times k} & 0\\
	0&0
	\end{array}\right)
$$ 
where $U=(u_{ij})_{ij}$ is a $k\times k$ orthogonal matrix whose entries satisfy
\begin{equation}\label{eq:example_assumption}
|u_{ij}|\leq \frac{C}{\sqrt {k}}.
\end{equation}	
For example, when $k=2m$, one can consider the following matrix,  related to the discrete Fourier transform, which is given in a block form by 
$$
U=\frac{1}{\sqrt{m}}\left(\left[\begin{array}{cc} 
\cos(\frac{2\pi j\ell}{m}) & -\sin(\frac{2\pi j\ell}{m})\\
\sin(\frac{2\pi j\ell}{m}) & \ \ \cos(\frac{2\pi j\ell}{m})
\end{array}\right]\right)_{j,\ell=1}^m.
$$
The above matrix satisfies condition (\ref{eq:example_assumption}) for $C=\sqrt 2$.

For any such $A$, the orthogonality of $U$ implies that  $\|A\|_{\cL(\ell_\infty^N,\ell_1^N)}\leq k$ and $\sum_{j,\ell=1}^N a_{j\ell}^2 = k$. Therefore, 
\begin{equation*}
\sum_{j,\ell=1}^N |a_{j\ell}| \geq \frac{\sum_{j,\ell=1}^N a_{j\ell}^2}{\max_{j,\ell}|a_{j\ell}|}\geq \frac{k\sqrt k}{C} \geq \frac{1}{C}\cdot \|A\|_{\cL(\ell_\infty^N,\ell_1^N)}\cdot \sqrt{\rank A} , 
\end{equation*} 
where the middle inequality follows from the assumption (\ref{eq:example_assumption}).
\end{example}

\begin{cor}\label{cor:symp_ineq}
	For any finite collection of vectors $v_1,\dots,v_N$ in $\R^{2n}$, 
	\begin{equation}\label{eq:symp_ineq}
	\sum_{i,j=1}^N |\left<v_i, Jv_j\right>|\leq 3K_G \cdot\sqrt{2n}\cdot \max_{|t_i|,|s_j|\leq 1}
	\left<\sum_{i=1}^N t_iv_i, J\sum_{j=1}^N s_j v_j\right>.
	\end{equation}
\end{cor}
\begin{proof}
	Consider the matrix $A=(a_{ij})_{i,j=1}^N$ defined by $a_{ij}:=\left<v_i, Jv_j\right>$. Then, $\rank A\leq 2n$ and 
	\begin{equation*}
	\|A\|_{\cL(\ell_\infty^N,\ell_1^N)} = \max_{|t_i|,|s_j|\leq 1}
	\sum_{i,j=1}^Nt_i s_j\left< v_i, Jv_j\right>.
	\end{equation*}
	Applying Theorem~\ref{thm:main_ineq} to the matrix $A$ gives the desired inequality.
\end{proof}

The next result gives an upper-bound for the infimum of the sum of norms of vectors $v_1,\dots, v_N\in \R^{2n}$ under the action of $\Sp(2n)$, by means of the rank and the $\cL(\ell_\infty^N, \ell_1^N)$-norm of the matrix $(\left<v_i, Jv_j\right>)_{i,j=1}^N$. We remark that this matrix is invariant under the action of $\Sp(2n)$ on the vectors $\{v_i\}_{i=1}^N$ (as follows easily from (\ref{eq:symp_matrix})).
\begin{theorem}\label{thm:norm_under_Sp}
	Let $v_1,\dots,v_N\in\R^{2n}$ and consider the $N\times N$ matrix defined by $A:=\left(\left<v_i, Jv_j\right>\right)_{i,j=1}^N$. Then, 
	\begin{equation}\label{eq:norm_under_Sp}
	\inf_{S\in\Sp(2n)}\left(\sum_{i=1}^N \|Sv_i\|_{\ell_2^{2n}}\right)^2\leq 3K_G\cdot\rank A\cdot \|A\|_{\cL(\ell_\infty^N,\ell_1^N)}.
	\end{equation}
\end{theorem}
\begin{proof}
	By homogeneity, we may assume that $\|A\|_{\cL(\ell_\infty^N,\ell_1^N)}\leq 1$. Denoting by $V$ the $2n\times N$ matrix whose columns are the vectors $\{v_i\}_{i=1}^N$, we can write
	\begin{equation}\label{eq:A_vs_J}
	A=V^TJ_{2n}V.
	\end{equation}
	We split the proof into cases, with respect to the rank of $A$:
	\begin{enumerate}
		\item[Case 1:] Assume that $\rank A = 2n$, then $V$ must be of full rank, which means that the vectors $\{v_i\}_{i=1}^N$ span $\R^{2n}$. By Lemma~\ref{lem:our_GT}, there exists a vector $\Lambda=(\lambda_1,\dots,\lambda_N)\in\R^N$ such that $\lambda_i>0$ for all $i$, $\|\Lambda\|_{\ell_2^N}\leq1$, and such that the matrix $B:=D_\Lambda^{-1}AD_\Lambda^{-1}$ satisfies $\|B\|_{\cL(\ell_2^N,\ell_2^N)}\leq 3K_G$. 
		Since $J$ is an anti-symmetric operator, it follows from (\ref{eq:A_vs_J}) that
		\begin{equation*}
		B^T=D_\Lambda^{-1}V^TJ_{2n}^T V D_\Lambda^{-1} = -B,
		\end{equation*}
		namely, $B$ is an anti-symmetric matrix of rank $2n$. By the spectral theorem there exists a $2n\times N$ matrix $Q$, which is a part of an $N\times N$ orthogonal matrix, such that $B=Q^TRQ$, where $R$ is a $2n\times 2n$ matrix of the form 
		\begin{equation*}
		R=\left(\begin{array}{ccccc}
		0 & -\mu_1 & & & \\
		\mu_1 & 0 & & &\\
		& & \ddots & &\\
		& & & 0 & -\mu_n \\
		& & & \mu_n & 0 
		\end{array}\right),
		\end{equation*}
		for some $0<\mu_i\leq 3K_G$.
		Denoting $$
		M:=(\sqrt{\mu_1},\sqrt{\mu_1},\sqrt{\mu_2},\sqrt{\mu_2},\dots,\sqrt{\mu_n},\sqrt{\mu_n})\in\R^{2n},$$
		we have $R=D_{{M}}R_0D_{{M}}$ where
		\begin{equation*}
		R_0=\left(\begin{array}{ccccc}
		0 & -1 & & & \\
		1 & 0 & & &\\
		& & \ddots & &\\
		& & & 0 & -1 \\
		& & & 1 & 0 
		\end{array}\right),
		\end{equation*}
		and is equal, up to a change of order of the basis elements, to $J$. Namely, $R_0 = P^TJ P$, where $P$ is a permutation matrix.
		We conclude that 
		\begin{equation*}
		A = D_\Lambda B D_\Lambda = W^T J W,
		\end{equation*}
		where $W$ is the $2n\times N$ matrix defined by
		\begin{equation*}
		W:= P\cdot D_{{M}}\cdot Q\cdot D_\Lambda.
		\end{equation*}
		By (\ref{eq:A_vs_J}) we obtain
		\begin{equation}\label{eq:V_vs_W}
		V^TJV=A=W^TJW.
		\end{equation}
		Let us show that $\ker V=\ker W$. Indeed, $v\in \ker V$ if and only if for any $u\in \R^{2n}$, $\left<u,Vv\right>=0$. Since $J$ is invertible, this is equivalent to $\left<u,JVv\right>=0$ for all $u\in \R^{2n}$. Moreover, $V$ is of full rank and so its image is $\R^{2n}$. Therefore, the latter condition is equivalent to $\left<w,V^TJVv\right>=\left<Vw,JVv\right>=0$ for all $w\in\R^N$. We conclude that $\ker V=\ker V^T JV$. Arguing the same for $W$ and using (\ref{eq:V_vs_W}) yields $\ker V=\ker W$. 
		
		Now, as both matrices $V,W$ are of full rank and have the same kernel, there exists a $2n\times 2n$ matrix $S$ such that 
		\begin{equation*}
		W=SV.
		\end{equation*}
		Plugging this back in (\ref{eq:V_vs_W}) yields
		\begin{equation*}
		V^T S^T J S V = V^T J V,
		\end{equation*}
		which implies (since $V$ is of full rank) that $S^TJS = J$ and hence $S\in\Sp(2n)$.
		Finally, let us bound the sum of norms of the vectors $\{Sv_i\}_{i=1}^N$. Denoting by $\{e_i\}_{i=1}^N$ the standard basis of $\R^N$, we have $Sv_i=We_i$. In addition, 
		\begin{eqnarray*}
			\|We_i\|_{\ell_2^{2n}}&=& \|P\cdot D_M\cdot Q\cdot D_\Lambda e_i\|_{\ell_2^{2n}}\\
			&\leq & \| P\|_{\cL(\ell_2^{2n},\ell_2^{2n})}\cdot \|D_{{M}}\|_{\cL(\ell_2^{2n},\ell_2^{2n})} \cdot \|Qe_i\|_{\ell_2^{2n}}\cdot |\lambda_i|.
		\end{eqnarray*}
		Since $P$ is an orthogonal matrix, $\| P\|_{\cL(\ell_2^{2n}, \ell_2^{2n})}= 1$. In addition, 
		$$\|D_{{M}}\|_{\cL(\ell_2^{2n},\ell_2^{2n})} = \max_i \sqrt{\mu_i} \leq \sqrt{3K_G}.$$ 
		Therefore,
		\begin{eqnarray*}
			\sum_{i=1}^N\|Sv_i\|_{\ell_2^{2n}} &\leq& \sqrt{3K_G}\cdot \sum_{i=1}^N \|Qe_i\|_{\ell_2^{2n}}\cdot |\lambda_i|\\
			&\leq& \sqrt{3K_G}\cdot \left(\sum_{i=1}^N \|Qe_i\|_{\ell_2^{2n}}^2\right)^\frac{1}{2}\cdot \left(\sum_{i=1}^N \lambda_i^2\right)^\frac{1}{2}\\
			&\leq& \sqrt{3K_G}\cdot \left(\sum_{i=1}^N \|Qe_i\|_{\ell_2^{2n}}^2\right)^\frac{1}{2} \\
			&\leq& \sqrt{3K_G}\cdot \sqrt{2n},			
		\end{eqnarray*}
		where the last inequality follows from the fact that $Q$ is a part of an $N\times N$ orthogonal matrix. Since $2n=\rank A$, this concludes the proof of the theorem for this case.
		
		\item[Case 2:] Assume that $\rank A<2n$. Set $E:=\Span\{v_1,\dots, v_N\}\subset\R^{2n}$. Denote by $ \ell $ the dimension of the kernel of the restriction of the
		bilinear form $\left<\cdot ,J\cdot\right>$ to $ E $. By a well known fact from symplectic linear algebra, $ dim E - \ell =: 2k $ is even, we 
		have $ k + \ell \leqslant n $, and moreover there exists a linear 
		symplectic matrix $T\in \Sp(2n)$ such that $ T(E) = span\{e_1,\dots,e_k,e_{n+1},\ldots, e_{n+k+\ell}\}$. Since both sides of the inequality (\ref{thm:norm_under_Sp})
		are invariant under the action of $ \Sp(2n) $ on vectors $ v_1, \ldots, v_N $, we may assume without loss of generality that 
		$$ E = span\{e_1,\dots,e_k,e_{n+1},\ldots, e_{n+k+\ell}\} .$$ Denote $$ E_0 = span\{e_1,\dots,e_k,e_{n+1},\ldots, e_{n+k}\},$$ $$ E_1 = span\{e_{n+k+1},\ldots, e_{n+k+\ell}\},$$
		$$ E_2 = span\{e_{k+1},\ldots, e_{k+\ell}\} ,$$ $$ E_3 = span\{e_{k+\ell+1},\dots,e_n,e_{n+k+\ell+1},\ldots, e_{2n}\} .$$ We have the orthogonal decompositions 
		$ E = E_0\oplus E_1$ and $\R^{2n} = E_0\oplus E_1\oplus E_2\oplus E_3$. 
		Consider  the orthogonal projections $\pi_i:\R^{2n}\rightarrow E_i$ for $i=0,1,2,3$. 
		For any $u,v\in E$ we have 
		\begin{equation*}
		\left<v, Ju\right> = \left<\pi_0v, J\pi_0 u\right>.
		\end{equation*}
		Therefore, the matrix $A$ defined by $a_{ij} = \left<v_i, Jv_j\right>$ does not change when we replace $\{v_i\}$ by $\{\pi_0v_i\}$. Thus
		$$
		\rank A = \rank \left<\cdot, J\cdot\right>|_{E_0} = \dim E_0 = 2k
		$$ and we may apply Case 1 to the vectors $\{\pi_0 v_i\}$ in $(E_0, J|_{E_0})$. We conclude that there exists a matrix $S_0:E_0\rightarrow E_0$ such that $S_0^TJ|_{E_0} S_0 =  J|_{E_0} $ and
		\begin{equation*}
		\sum_{i=1}^N\|S_0\pi_0v_i\|_{\ell_2^{2k}} \leq \sqrt{3K_G}\cdot \sqrt{2k}.
		\end{equation*}
		For any given $\varepsilon>0$ consider the operator
		\begin{equation*}
		S_\varepsilon:= 
		S_0 \pi_0 + \varepsilon \pi_1+\frac{1}{\varepsilon}\pi_2+ \pi_3.
		\end{equation*}
		One can check that $S_\varepsilon$ 
		satisfies relation (\ref{eq:symp_matrix}) and therefore belongs to $\Sp(2n)$. Finally, the sum of norms of $\{S_\varepsilon v_i\}$ is bounded as follows: 
		\begin{eqnarray*}
			\sum_{i=1}^N \|S_\varepsilon v_i\|_{\ell_2^{2n}} &=& \sum_{i=1}^N \|S_\varepsilon \pi_0v_i + S_\varepsilon \pi_1 v_i\|_{\ell_2^{2n}}\\
			&\leq& \sum_{i=1}^N \|S_0 \pi_0v_i \|_{\ell_2^{2n}} + \sum_{i=1}^N \|\varepsilon\cdot 
			\pi_1 v_i\|_{\ell_2^{2n}}\\ 
			&\leq& \sqrt{3K_G}\cdot \sqrt{2k} + \varepsilon\sum_{i=1}^N\|v_i\|_{\ell_2^{2n}}.
		\end{eqnarray*}
		Taking $\varepsilon\rightarrow0$ we conclude that 
		\begin{equation*}
		\inf_{S\in\Sp(2n)} \sum_{i=1}^N \|S v_i\|_{\ell_2^{2n}}  \leq  \sqrt{3K_G}\cdot \sqrt{2k} = \sqrt{3K_G}\cdot \sqrt{\rank A}.
		\end{equation*} 	
	\end{enumerate}
\end{proof}

\begin{example}
	Let us show that for any $k\leq 2n$ there exist vectors $v_1,\dots,v_N\in\R^{2n}$ with $\dim span\{v_1,\dots, v_N\}\leq k$ such that for any matrix $S\in\Sp(2n)$,
	\begin{equation*}
	\left(\sum_{j=1}^N \|Sv_j\|_{\ell_2^{2n}}\right)^2\geq C\cdot k\cdot \|A\|_{\cL(\ell_\infty^N,\ell_1^N)},
	\end{equation*}
	where $A:=\left(\left<v_i, J_{2n}v_j\right>\right)_{i,j=1}^N$. Notice that it is enough to consider the case where $k=2m$ is even. In this case, take the vectors $v_j$ to be zero for $j>k=2m$ and $\{v_1,\dots,v_{2m}\} = \{e_1, \dots,e_m, e_{n+1},\dots ,e_{n+m}\}$. The corresponding matrix is
	$$
	A=\left(\begin{array}{cc}
	J_{2m}&0\\0&0
	\end{array}\right)
	$$ 
	and $\left(\sum_{j=1}^N\|v_j\|_{\ell_2^{2n}}\right)^2 = k^2=\rank A\cdot \|A\|_{\cL(\ell_\infty^N,\ell_1^N)}$. Let $S\in\Sp(2n)$ and set $w_j:= S v_j$ for all $1\leq j\leq N$. Denote by $W$ the $2n\times N$ matrix whose columns are the vectors $w_j$. Then, it follows from (\ref{eq:symp_matrix}) that $A = W^T J_{2n} W$. 
	Let $\tilde W $ be the $2n\times 2m$ matrix whose columns are the first $2m$ columns of $W$, then $\tilde W^T J_{2n} \tilde W = J_{2m}$. 
	
	Given a $k\times\ell$ matrix $P$ whose columns are the vectors $\{p_1,\dots,p_\ell\}\subset \R^k$, we denote by $\Pi(P):=\prod_{i=1}^\ell\|p_i\|_{\ell_2^{k}}$ the product of the Euclidean norms of the columns of $P$. A generalized version of Hadamard's inequality states that for any pair of $k\times \ell$ matrices $P$ and $Q$, the determinant of $P^TQ$ is bounded by the product of $\ell_2$-norms of the columns of $P$ and $Q$, $\det(P^T Q)\leq \Pi(P)\cdot \Pi(Q)$. Applying this to $P=\tilde W$, $Q=J_{2n}\tilde W$, we have 	 
	\begin{equation*}
	1 = \det(J_{2m}) = \det(\tilde W^T \cdot J_{2n}\tilde W)\leq \Pi(\tilde W)\cdot \Pi(J_{2n}\tilde W) = \Pi(\tilde W)^2.
	\end{equation*}
	Using the inequality of arithmetic and geometric means we conclude that 
	$$
	1\leq \prod_{j=1}^{2m}\|w_j\|_{\ell_2^{2n}} \leq \frac{1}{2m}\sum_{j=1}^{2m} \|w_j\|_{\ell_2^{2n}},
	$$ 
	and so 
	$$
	\sum_{j=1}^{N} \|w_j\|_{\ell_2^{2n}} = \sum_{j=1}^{2m} \|w_j\|_{\ell_2^{2n}}\geq 2m=k=\sqrt{\rank A\cdot \|A\|_{\cL(\ell_\infty^N,\ell_1^N)}}.
	$$ 
\end{example}
\begin{rem}
	In the above example we actually proved the following stronger statement: For any collection of vectors $v_1,\dots,v_N\in \R^{2n}$ that satisfy \begin{equation*}
	V^T J_{2n} V = \left(\begin{array}{cc}
	J_{2m}&0\\0&0
	\end{array}\right)=:A
	\end{equation*}
	for some $m\leq n$, where $V$ is the matrix whose columns are $\{v_j\}_{j=1}^N$, we have
	\begin{equation}
	\sum_{j=1}^N\|v_j\|_{\ell_2^N} \geq 2m = \sqrt{\rank(A)\cdot\|A\|_{\cL(\ell_\infty^N,\ell_1^N)} }.
	\end{equation}
\end{rem}

\subsection*{Acknowledgements.}
The authors are grateful to the Creator for giving them the understanding presented in this paper. We also thank Lev Buhovsky for useful discussions and comments.
Shira Tanny extends her special thanks to Lev Buhovsky and Leonid Polterovich for their mentorship and guidance. 

S.T. was partially supported by ISF Grant 2026/17 and by the Levtzion Scholarship.

\bibliographystyle{abbrv}
\bibliography{refs}
\vspace*{1cm}

\newpage

\paragraph{Efim Gluskin,}$ $\\ 
School of Mathematical Sciences\\ 
Tel Aviv University \\
Ramat Aviv, Tel Aviv 69978\\
Israel\\
E-mail: gluskin@tauex.tau.ac.il

\paragraph{Shira Tanny,}$ $\\
School of Mathematical Sciences\\ 
Tel Aviv University \\
Ramat Aviv, Tel Aviv 69978\\
Israel\\
E-mail: tanny.shira@gmail.com

\end{document}